\documentclass[reqno,a4paper,12pt]{amsart}
\usepackage[top=32mm, bottom=32mm, left=32mm, right=32mm]{geometry}
\usepackage{mathrsfs,latexsym,float,enumerate,amssymb,amsmath,bm,color,lineno}
\usepackage{amscd,verbatim}
\usepackage{graphicx}
\usepackage[colorlinks=true]{hyperref}

\numberwithin{equation}{section}

\newtheorem{theorem}{Theorem}
\newtheorem{lemma}[theorem]{Lemma}

\theoremstyle{definition}
\newtheorem{definition}[theorem]{Definition}
\newtheorem{proposition}[theorem]{Proposition}

\newtheorem{question}[theorem]{Question}

\theoremstyle{remark}
\newtheorem{remark}[theorem]{Remark}

\begin{document}

\title[]{Answering an Open Problem on $T$-Norms for Type-2 Fuzzy Sets}

%    Information for first author
\author[X. Wu]{Xinxing Wu}
\address[X. Wu]{School of Sciences, Southwest Petroleum University, Chengdu, Sichuan 610500, China}
\email{wuxinxing5201314@163.com}

\author[G. Chen]{Guanrong Chen}
\address[G. Chen]{Department of Electrical Engineering, City University of Hong Kong,
Hong Kong SAR, China}
\email{gchen@ee.cityu.edu.hk}

\thanks{This work was supported by the National Natural Science Foundation of China
(No. 11601449), the Science and Technology Innovation Team of Education Department of Sichuan for Dynamical Systems and its
Applications (No. 18TD0013), and the Youth Science and Technology Innovation Team of Southwest
Petroleum University for Nonlinear Systems (No. 2017CXTD02).}

%    General info
%%\subjclass[2010]{03E72, 54H20.}

\date{\today}%{xxx xx, 201x and, in revised form, xxx xx, 201x.}

%\dedicatory{This paper is dedicated to our advisors.}


\begin{abstract}
This paper proves that a binary operation ${\star}$ on ${[0, 1]}$, ensuring that the binary operation
${\curlywedge}$ is a ${t}$-norm or ${\curlyvee}$ is a ${t}$-conorm, is a ${t}$-norm,
where ${\curlywedge}$ and ${\curlyvee}$ are special convolution operations defined by
$$
{(f\curlywedge g)(x)=\sup\left\{f(y)\star g(z): y\vartriangle z=x\right\},}
$$
$$
{(f\curlyvee g)(x)=\sup\left\{f(y)\star g(z): y\ \triangledown\ z=x\right\},}
$$
for any ${f, g\in Map([0, 1], [0, 1])}$, where ${\vartriangle}$ and ${\triangledown}$ are
a continuous ${t}$-norm and a continuous ${t}$-conorm on ${[0, 1]}$, answering negatively an open
problem posed in \cite{HCT2015}. Besides, some characteristics of ${t}$-norm and ${t}$-conorm
are obtained in terms of the binary operations ${\curlywedge}$ and ${\curlyvee}$.
\end{abstract}

\maketitle

\section{Introduction}

In 1975, Zadeh~\cite{Z1975} introduced the notion of type-2 fuzzy sets (T2FSs) -- that is, fuzzy set with fuzzy sets as truth values
(simply, ``fuzzy-fuzzy sets'') -- being an extension of type-1 fuzzy sets (FSs) and interval-valued fuzzy sets (IVFSs), which was also equivalently
expressed in different forms by Mendel {\it et al.} (\cite{KM2001}--\cite{MJ2002}). Because the truth values of T2FSs are fuzzy, they are more adaptable to a
further study of uncertainty than FSs and have been applied in many studies (\cite{M2007}--\cite{TLSR2015}).
Mendel \cite{M2007} summarized some important advances for FSs and T2FSs from 2001 to 2007.
Hu and Kwong~\cite{HK2013} discussed $t$-norm operations of T2FSs and obtained a few properties of type-2 fuzzy numbers.
For better understanding of T2FSs, Aisbett {\it et al.} \cite{ARM2010} translates their constructs to the language of functions in spaces.
Chen and Wang \cite{CW2013} used T2FSs to give a new technique for fuzzy multiple attributes decision making. Sola {\it et al.}~\cite{S2015} provided
a more general perspective for interval T2FSs and showed that IVFSs can be viewed as a special case of interval T2FSs.
Ruiz {\it et al.}~\cite{RH2016} obtained two results for join and meet operations for T2FSs with arbitrary secondary
memberships. Recently, Wang ~\cite{W2017} introduced the notion of conditional fuzzy sets to characterize T2FSs.
Then, Wu~{\it et al.} \cite{W2019} presented a Jaccard similarity measure for general T2FSs, as
an extension of the Jaccard similarity measure for FSs and IVFSs.

\medskip

Being an extension of the logic connective conjunction and disjunctionin classical two-valued logic,
triangular norms ($t$-norms) with the neutral $1$ and triangular conorms ($t$-conorms) with the neutral $0$
on the unit interval $I=[0, 1]$ were introduced by Menger \cite{Me1942} in 1942 and by Schweizer and Sklar \cite{SS1961}
in 1961, respectively. Because $t$-norms and $t$-conorms have a close connection with fuzzy set theory
and order related theories, they play an important role in many fields, such as fuzzy set theory \cite{Zi2001}, fuzzy logic \cite{BJ2008},
fuzzy systems modeling \cite{Y2001}, and probabilistic metric spaces \cite{SS1961}. Walker and Walker \cite{WW2006}
extended $t$-norms and $t$-conorms to the algebra of truth values of T2FSs. Then, Hern\'{a}ndes {\it et al.}~\cite{HCT2015} introduced the
notions of $t_{r}$-norm and $t_{r}$-conorm by adding some ``restrictive axioms'' (see Definition~\ref{Def-1} below) with
systematic analysis. In particular, they \cite{HCT2015} proved that the following binary operation $\curlywedge$
(resp., $\curlyvee$) on the set of all normal and convex functions constructed by convolution is a $t_{r}$-norm (resp., a $t_{r}$-conorm).
Recently, we proved \cite{WC2019} that the fuzzy metric $M$ of every stationary fuzzy metric space $(X, M, \star)$ is uniformly continuous.

\medskip

Throughout this paper, let $I=[0, 1]$, $Map(X, Y)$ be the set of all mappings
from $X$ to $Y$, and `$\leq$' denote the usual order relation in the lattice of real numbers.
In particular, let $\mathbf{M}=Map(I, I)$ and $\mathbf{L}$ be the set of all normal and convex functions in $\mathbf{M}$.

\begin{definition}\cite{KMP2000}
A {\it $t$-norm} on $I$ is a binary operation $\star: I^{2}\rightarrow I$ satisfying
\begin{enumerate}%%[(T1)]
  \item[(T1)] ({\it commutativity}/{\it symmetry}) $x\star y=y\star x$ for $x, y\in I$;
  \item[(T2)] ({\it associativity}) $(x \star y) \star z=x\star (y\star z)$ for $x, y, z\in I$;
  \item[(T3)] ({\it increasing}) $\star$ is increasing in each argument;
  \item[(T4)] ({\it neutral element}) $1\star x=x\star 1=x$ for $x\in I$.
\end{enumerate}
A binary operation $\star: I^2\rightarrow I$ is a
{\it $t$-conorm} on $I$ if it satisfies axioms (T1), (T2), and (T3) above;
axiom (T4'): $0\star x=x\star 0=x$ for $x\in I$.
\end{definition}

For any subset $B$ of $X$, a special fuzzy set $\chi_{B}$, which is called the {\it characteristic function}
of $B$, is defined as
$$
\chi_{B}(x)=\begin{cases}
1, & x\in B,\\
0, & x\in X\setminus B.
\end{cases}
$$

\begin{definition}\cite{HCT2015}\label{Def-1}
A binary operation $T: \mathbf{L}^2 \rightarrow \mathbf{L}$ is a {\it $t_{r}$-norm}
($t$-norm according to the restrictive axioms), if
\begin{enumerate}[(O1)]
  \item $T$ is commutative, i.e., $T(f, g)=T(g, f)$ for $f, g\in \mathbf{L}$;
  \item $T$ is associative, i.e., $T(T(f, g), h)=T(f, T(g, h))$ for $f, g, h\in \mathbf{L}$;
  \item $T(f, \chi_{\{1\}})=f$ for $f\in \mathbf{L}$ (neutral element);
  \item letting $f, g, h\in \mathbf{L}$ such that $g\sqsubseteq h$;
  then, $T(f, g)\sqsubseteq T(f, h)$ (increasing in each argument);
  \item $T(\chi_{[0, 1]}, \chi_{[a, b]})=\chi_{[0, b]}$;
  \item $T$ is closed on $\mathbf{J}$;
  \item $T$ is closed on $\mathbf{K}$;
\end{enumerate}
where $\mathbf{J}$ is the set of all characteristic functions of the elements of $I$,
and $\mathbf{K}$ is the set of all characteristic functions of the
closed subintervals of $I$, i.e., $\mathbf{J}=\{\chi_{\{x\}}: x\in I\}$,
$\mathbf{K}=\{\chi_{[a, b]}: 0\leq a \leq b\leq 1\}$.

A binary operation $S: \mathbf{L}^2\rightarrow \mathbf{L}$ is a
{\it $t_r$-conorm} if it satisfies axioms (O1), (O2), (O4), (O6), and (O7) above;
axiom (O3'): $S(f, \chi_{\{0\}})=f$; and axiom (O5'):
$S(\chi_{[0, 1]}, \chi_{[a, b]})=\chi_{[a, 1]}$. Axioms (O1), (O2), (O3), (O3'),
and (O4) are called ``{\it basic axioms}'', and an operation that complies with
these axioms will be referred to as {\it $t$-norm} or {\it $t$-conorm}, respectively.
\end{definition}

Convolution as a standard way to combine functions was used to
construct operations on $Map(J, [0, 1])$. Let $\circ$ and $\blacktriangle$ be two binary operations
defined on $X$ and $Y$, respectively, and $\blacktriangledown$ be an appropriate operation on $Y$.
Define a binary operation $\bullet$ on the set $Map(X, Y)$ by
$$
(f\bullet g)(x)=\blacktriangledown \{f(y) \blacktriangle g(z): y\circ z=x\}.
$$
This method of defining an operation on $Map(X, Y)$ from operations on $X$ and $Y$ is called
{\it convolution}.

\begin{definition}\cite{HCT2015}
Let $\star$ be a binary operation on $I$, $\vartriangle$ be a
$t$-norm on $I$, and $\triangledown$ be a $t$-conorm on $I$. Define the binary operations
$\curlywedge$ and $\curlyvee: \mathbf{M}^2\rightarrow \mathbf{M}$ as follows: for $f, g\in \mathbf{M}$,
\begin{equation}\label{O-1}
(f\curlywedge g)(x)=\sup\left\{f(y)\star g(z): y\vartriangle z =x\right\},
\end{equation}
and
\begin{equation}\label{O-2}
(f\curlyvee g)(x)=\sup\left\{f(y)\star g(z): y\ \triangledown\ z =x\right\}.
\end{equation}
\end{definition}

In 2015, Hern\'{a}ndes {\it et al.}~\cite{HCT2015} proposed
the following open problem on the binary operations $\curlywedge$ and $\curlyvee$.

\begin{question}\label{Q-1}\cite{HCT2015}
{\it Apart from the $t$-norms, does there exist other binary operation
`$\star$' on $I$ such that `$\curlywedge$' and `$\curlyvee$' are, respectively,
a $t_r$-norm and a $t_r$-conorm on $\mathbf{L}$?}
\end{question}

This paper first gives a negative answer to Question~\ref{Q-1}, proving that, if a binary operation $\star$
ensures that $\curlywedge$ is a $t_{r}$-norm on $\mathbf{L}$ or $\curlyvee$ is a $t_{r}$-conorm on $\mathbf{L}$,
then $\star$ is a $t$-norm, i.e., $\star$ satisfies axioms (T1)--(T4). Then, it is proved that the following
are equivalent:
\begin{enumerate}[(1)]
  \item $\star$ is a $t$-norm on $I$;
  \item $\curlywedge$ is a $t_{r}$-norm on $\mathbf{L}$;
  \item $\curlywedge$ is a $t$-norm on $\mathbf{L}$;
  \item $\curlyvee$ is a $t_{r}$-conorm on $\mathbf{L}$;
  \item $\curlyvee$ is a $t$-conorm on $\mathbf{L}$.
\end{enumerate}
Finally, analogous results on $\vartriangle$ are presented when
the binary operation $\star$ is restricted to be a continuous $t$-norm.

\section{Preliminaries}
A {\it type-1 fuzzy set} $A$ in space $X$ is a mapping from $X$ to $I$, i.e.,
$A\in Map(X, I)$, and $A(x)$ is called the {\it degree of membership} of an element $x\in X$ to the set $A$.
The two sets $\O$ and $X$ are special elements in $Map(X, I)$, with $\O(x)\equiv 0$ and $X(x)\equiv 1$, respectively.
A fuzzy set $A\in Map(X, I)$ is {\it normal} if $\sup\{A(x): x\in I\}=1$.

\begin{definition}\cite{HCT2015}
A function $f\in \mathbf{M}$ is {\it convex} if, for any $x\leq y\leq z$, it holds that
$f(y)\geq f(x)\wedge f(z)$.
\end{definition}

\begin{definition}\cite{WW2005}
A {\it type-2 fuzzy set} $A$ in space $X$ is a mapping
$$
A: X\rightarrow \mathbf{M},
$$
i.e., $A\in Map(X, \mathbf{M})$. For any $x\in X$, $A(x)$ is also called the {\it degree of membership}
of an element $x\in X$ to the set $A$.
\end{definition}

\begin{definition}\cite{WW2005}
The operations of $\sqcup$ (union), $\sqcap$ (intersection), $\neg$ (complementation) on $\mathbf{M}$ are
defined as follows: for any $f, g\in \mathbf{M}$,
$$
(f\sqcup g)(x)=\sup\{f(y)\wedge g(z): y\vee z=x\},
$$
$$
(f\sqcap g)(x)=\sup\{f(y)\wedge g(z): y\wedge z=x\},
$$
and
$$
(\neg f)(x)=\sup\{f(y): 1-y=x\}=f(1-x).
$$
\end{definition}

From \cite{WW2005}, it follows that $\mathbb{M}=(\mathbf{M}, \sqcup, \sqcap, \neg, \chi_{\{0\}}, \chi_{\{1\}})$
does not have a lattice structure, although $\sqcup$ and $\sqcap$ satisfy the De Morgan's laws with respect
to the given operation $\neg$.

Walker and Walker \cite{WW2005} introduced the following partial order on $\mathbf{M}$.
\begin{definition}\cite{WW2005}
$f\sqsubseteq g$ if $f\sqcap g=f$; $f\preceq g$ if $f\sqcup g=g$.
\end{definition}

It follows from \cite[Proposition 14]{WW2005} that both $\sqsubseteq$ and $\preceq$ are partial
orders on $\mathbf{M}$. In \cite{HWW2010,HWW2008,WW2005}, it was proved that the subalgebra
$\mathbb{L}=(\mathbf{L}, \sqcup, \sqcap, \neg, \chi_{\{0\}}, \chi_{\{1\}})$
is a bounded complete lattice. In particular, $\chi_{\{0\}}$ and $\chi_{\{1\}}$
are the minimum and maximum, respectively.

For $f\in \mathbf{M}$, define $f^{L}$ and $f^{R}$ in $\mathbf{M}$ by
$$
f^{L}(x)=\sup\{f(y): y\leq x\},
$$
and
$$
f^{R}(x)=\sup\{f(y): y\geq x\}.
$$
Clearly, $f^L$ and $f^R$ are monotonically increasing and
decreasing, respectively. The following properties of $f^{L}$
and $f^{R}$ are obtained by Walker {\it et al.} (\cite{HWW2010,HWW2008,WW2005}).
\begin{proposition}\cite{WW2005}
For $f, g\in \mathbf{M}$,
\begin{enumerate}[(1)]
  \item $f\sqsubseteq g$ if and only if $f^{R}\wedge g \leq f\leq g^{R}$;
  \item $f\preceq g$ if and only if $f\wedge g^{L}\leq g\leq f^{L}$;
  \item $f\leq f^{L}$, $f\leq f^{R}$;
  \item $(f^{L})^{L}=f^{L}$, $(f^{R})^{R}=f^{R}$;
  \item $(f^{L})^{R}=(f^{R})^{L}=\sup f$;
\end{enumerate}
\end{proposition}

\begin{theorem}(\cite{HWW2010,HWW2008})\label{order-theorem}
Let $f, g\in \mathbf{L}$. Then, $f\sqsubseteq g$ if and only if
$g^L\leq f^L$ and $f^R\leq g^R$.
\end{theorem}

\begin{lemma}\label{L-R}
For $f\in \mathbf{M}$, $f^{L}(0)=f(0)$ and $f^{R}(1)=f(1)$.
\end{lemma}

\begin{proof}
From the definitions of $f^{L}$ and $f^{R}$, this holds trivially.
\end{proof}

\section{Answer to the Open Problem}

\subsection{Commutativity and Associativity of $\star$}\label{S3-1}

\begin{lemma}\label{1-Lemma}
Let $\star$ be a $t$-norm on $I$. Then, $x\star y=1$ if and only if $x=y=1$.
\end{lemma}

\begin{lemma}\label{Con-1}
Let $\vartriangle$ be a continuous $t$-norm on $I$ and $\star$ be a binary operation on $I$.
Then,
$$
(f\curlywedge g)(1)=f(1)\star g(1).
$$
\end{lemma}

\begin{proof}
Since $\vartriangle$ is a $t$-norm, from Lemma~\ref{1-Lemma}, it follows that
$$
(f\curlywedge g)(1)=\sup\{f(y)\star g(z): y\vartriangle z=1\}=f(1)\star g(1).
$$
\end{proof}

\begin{proposition}\label{Commu-Asso-Thm}
Let $\vartriangle$ be a continuous $t$-norm on $I$ and $\star$ be a binary operation on $I$.
Then, \begin{enumerate}[(1)]
  \item $\curlywedge$ and $\curlyvee$ are commutative on $\mathbf{L}$ if and only if $\star$ is commutative;
  \item If $\curlywedge$ and $\curlyvee$ are associative on $\mathbf{L}$, then $\star$ is
associative.
\end{enumerate}
\end{proposition}
\begin{proof}
(1) The sufficiency follows from the proof of \cite[Proposition~1]{HCT2015}. It remains to prove the necessity.
Suppose on the contrary that $\star$ is not commutative. Then, there exist $u, v\in I$ such that $u\star v\neq v\star u$.
Choose two functions $f, g\in \mathbf{M}$ as
$$
f(x)=(u-1)x+1,
$$
and
$$
g(x)=(v-1)x+1,
$$ for $x\in I$. It can be verified that
$f, g\in \mathbf{L}$, as $f$ and $g$ are decreasing.
Since $\curlywedge$ is commutative, applying Lemma~\ref{Con-1} yields that
\begin{align*}
&u\star v=f(1) \star g(1)=(f\curlywedge g)(1)\\
&=(g\curlywedge f)(1)=g(1) \star f(1)=v\star u,
\end{align*}
which is a contradiction. Therefore,
$\star$ is commutative.

\medskip

(2) Suppose on the contrary that $\star$ is not associative. Then, there exist
$u, v, w\in I$ such that $u\star (v\star w)\neq (u\star v) \star w$. Choose three
functions $f, g, h\in \mathbf{M}$ as
$$
f(x)=(u-1)x+1,
$$
$$
g(x)=(v-1)x+1,
$$
and
$$
h(x)=(w-1)x+1,
$$ for $x\in I$.
It can be verified that $u, v, w\in \mathbf{L}$, as $f$, $g$, and $h$ are decreasing. Since
$\curlywedge$ is associative, applying Lemma~\ref{Con-1} yields that
\begin{align*}
&u\star (v\star w)=f(1) \star (g\curlywedge h)(1)=(f\curlywedge (g\curlywedge h))(1)\\
&=
((f\curlywedge g) \curlywedge h)(1)=(f\curlywedge g)(1) \star h(1)=(u\star v)\star w,
\end{align*}
which is a contradiction. Therefore,
$\star$ is associative.
\end{proof}

\begin{figure}[H]
\begin{center}
\scalebox{0.6}{\includegraphics{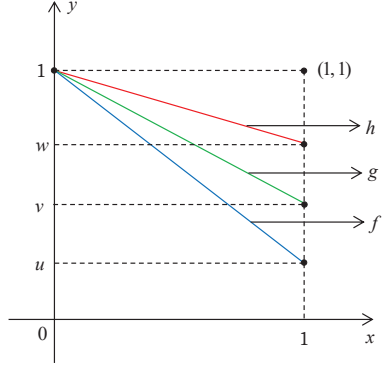}}
\renewcommand{\figure}{Fig.}
\caption{An illustration diagram of the functions $f, g, h$.
}
\end{center}
\end{figure}

\begin{remark}
Similar results to Proposition~\ref{Commu-Asso-Thm} are obtained by Hern\'{a}ndez {\it et al.}
\cite{HCT2015} under the assumption that $\curlywedge$ and $\curlyvee$ are commutative or associative
on $\mathbf{M}$, which is stronger than the condition in Proposition~\ref{Commu-Asso-Thm}.
\end{remark}

\subsection{Neutral Element $1$ for $\star$}

For any fixed $x\in I$, define $\mathscr{W}_{x}: I\rightarrow I$ by
$$
\mathscr{W}_{x}(t)=\begin{cases}
0, & t\in [0, x),\\
t, & t\in [x, 1],
\end{cases}
$$
for $t\in I$. It can be verified that $\mathscr{W}_{x}\in \mathbf{L}$, as $\mathscr{W}_{x}$ is increasing
for $x\in I$.

\begin{figure}[H]
\begin{center}
\scalebox{0.6}{\includegraphics{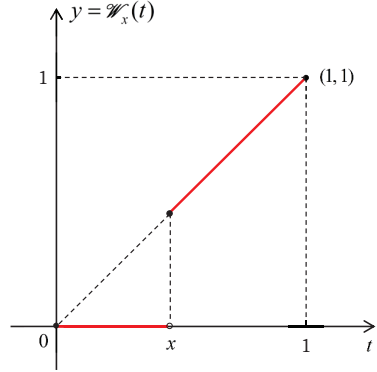}}
\renewcommand{\figure}{Fig.}
\caption{An illustration diagram of the function $\mathscr{W}_{x}(t)$.
}
\end{center}
\end{figure}

\begin{lemma}\label{0-element}
Let $\vartriangle$ be a continuous $t$-norm on $I$ and $\star$ be a binary operation on $I$.
If $\curlywedge$ is a $t$-norm on $\mathbf{L}$, then
$0\star x=x\star 0=0$ for all $x\in I$.
\end{lemma}

\begin{proof}
(1)
As $\chi_{\{1\}}$ is a neural element, by Lemma \ref{Con-1}, one has
\begin{align*}
0&=\chi_{\{0\}}(1)=(\chi_{\{1\}}\curlywedge \chi_{\{0\}})(1)\\
&=\chi_{\{1\}}(1) \star \chi_{\{0\}}(1)=1\star 0.
\end{align*}

(2) Fix any $x\in (0, 1)$. From
$\mathscr{W}_{x}(t)=(\mathscr{W}_{x}\curlywedge \chi_{\{1\}})(t)=\sup\{\mathscr{W}_{x}(y) \star \chi_{\{1\}}(z):
y \vartriangle z=t\}$, it follows that, for any $t\in (0, x)$,
\begin{equation}\label{e-1}
0=\mathscr{W}_{x}(t)=\sup\{\mathscr{W}_{x}(y) \star \chi_{\{1\}}(z):
y \vartriangle z=t\}.
\end{equation}
Since $\vartriangle$$(x, -)$ is continuous on $[0, 1]$, and $\vartriangle$$(x, 0)=0$, $\vartriangle$$(x, 1)=x$,
it follows from the intermediate value theorem that there exists some
$z_1\in (0, 1)$ such that $\vartriangle$$(x, z_1)=x$$\vartriangle$$z_1=t$. This, together with \eqref{e-1}, implies that
$$
0\geq \mathscr{W}_{x}(x)\star \chi_{\{1\}}(z_1)=x\star 0, \text{ i.e., } x\star 0=0.
$$

(3) Note that $0=\chi_{\{0\}}(\frac{1}{4})=(\chi_{\{0\}}\curlywedge \chi_{\{1\}})(\frac{1}{4})
=\sup\left\{\chi_{\{0\}}(y)\star \chi_{\{1\}}(z): y\vartriangle z=\frac{1}{4}\right\}$.
Similarly to the proof of (2), it follows that there exists
$y\in (0, 1)$ such that $y\vartriangle\frac{1}{2}=\frac{1}{4}$. This implies that
$$
0\geq \chi_{\{0\}}(y) \star \chi_{\{1\}}(\frac{1}{2})=0 \star 0, \text{ i.e., } 0=0\star 0.
$$

Summing up (1)--(3) and the commutativity of $\star$ (Proposition \ref{Commu-Asso-Thm}),
it follows that, for any $x\in [0, 1]$,
$$
x\star 0=0\star x=0.
$$
\end{proof}

\begin{lemma}\label{1-elemrnt}
Let $\vartriangle$ be a continuous $t$-norm on $I$ and $\star$ be a binary operation on $I$.
If $\curlywedge$ is a $t$-norm on $\mathbf{L}$, then
$1\star x=x\star 1=x$ for all $x\in I$.
\end{lemma}

\begin{proof}
(1) Since $\chi_{\{1\}}$ is a neural element, from Lemma \ref{Con-1}, it follows that
\begin{align*}
1&=\chi_{\{1\}}(1)=(\chi_{\{1\}}\curlywedge \chi_{\{1\}})(1)\\
&=\chi_{\{1\}}(1)\star \chi_{\{1\}}(1) =1\star 1.
\end{align*}

(2) For any fixed $x\in (0, 1)$, $x=\mathscr{W}_{x}(x)=(\mathscr{W}_{x}\curlywedge \chi_{\{1\}})(x)
=\sup\left\{\mathscr{W}_{x}(y)\star \chi_{\{1\}}(z): y\vartriangle z=x\right\}$. For $y, z\in I$ with $y\vartriangle z=x$,
consider the following two cases:

Case 1. If $z=1$, then $y=x$. This implies that $\mathscr{W}_{x}(y)\star \chi_{\{1\}}(z)=x\star 1$;

Case 2. If $z<1$, then $\chi_{\{1\}}(z)=0$. Applying Lemma~\ref{0-element} gives that
$$
\mathscr{W}_{x}(y)\star \chi_{\{1\}}(z)=0.
$$
Thus,
$$
x=\sup\left\{\mathscr{W}_{x}(y)\star \chi_{\{1\}}(z): y\vartriangle z=x\right\}=x\star 1.
$$

The proof is completed by applying $0\star 1=0$ and the commutativity of $\star$.
\end{proof}

\subsection{Increasing in Each Argument for $\star$}
For any fixed $x\in I$, define $\mathscr{V}_{x}: I\rightarrow I$ by
$$
\mathscr{V}_{x}(t)=(x-1)t+1, \quad \forall t\in I.
$$
It can be verified that $\mathscr{V}_{x}\in \mathbf{L}$, as $\mathscr{V}_{x}$
is decreasing for $x\in I$. Clearly, functions $f$, $g$, and $h$ constructed in Proposition~\ref{Commu-Asso-Thm}
satisfy that $f=\mathscr{V}_{u}$, $g=\mathscr{V}_{v}$, and $h=\mathscr{V}_{w}$.

\begin{figure}[H]
\begin{center}
\scalebox{0.6}{\includegraphics{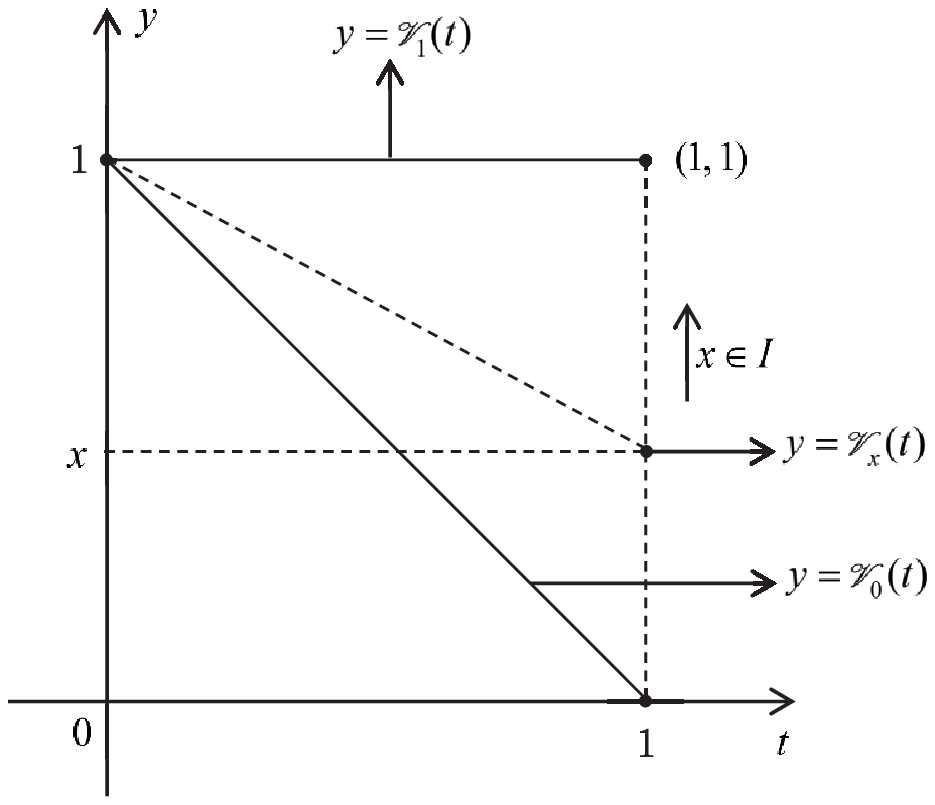}}
\renewcommand{\figure}{Fig.}
\caption{An illustration diagram of the function $\mathscr{V}_{x}(t)$.
}
\end{center}
\end{figure}

Applying the decreasing property of $\mathscr{V}_{x}$ immediately yields the following result.

\begin{lemma}\label{L-R-Lemma}
For any $x\in I$, $\mathscr{V}^{L}_{x}\equiv 1$ and $\mathscr{V}^{R}_{x}=\mathscr{V}_{x}$.
\end{lemma}

\begin{lemma}\label{Lemma-Q}
For any $x_1, x_2\in I$ with $x_1\leq x_2$,
$\mathscr{V}_{x_1}\sqsubseteq \mathscr{V}_{x_2}$.
\end{lemma}

\begin{proof}
Clearly, $\mathscr{V}_{x_1}\leq \mathscr{V}_{x_2}$. Applying Lemma~\ref{L-R-Lemma} yields that
$$
\mathscr{V}^{L}_{x_2}\leq \mathscr{V}^{L}_{x_1},
$$
and
$$
\mathscr{V}^{R}_{x_1}\leq \mathscr{V}^{R}_{x_2}.
$$
This, together with Theorem~\ref{order-theorem}, implies that
$$
\mathscr{V}_{x_1}\sqsubseteq \mathscr{V}_{x_2}.
$$
\end{proof}

\begin{figure}[H]
\begin{center}
\scalebox{0.6}{\includegraphics{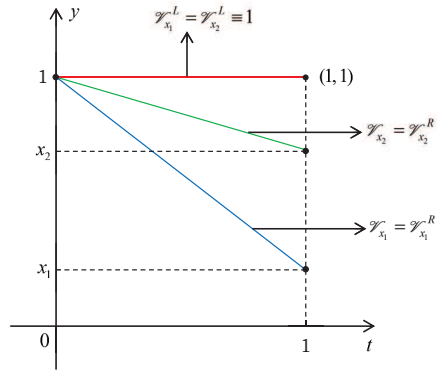}}
\renewcommand{\figure}{Fig.}
\caption{An illustration diagram of the function $\mathscr{V}_{x_1}(t)$ and $\mathscr{V}_{x_2}(t)$.
}
\end{center}
\end{figure}

\begin{lemma}\label{Increasing-Lemma}
Let $\vartriangle$ be a continuous $t$-norm on $I$ and $\star$ be a binary operation on $I$.
If $\curlywedge$ is a $t$-norm on $\mathbf{L}$, then for any $y\in (0, 1)$
the functions $\star^{r}_{y}$ and $\star^{l}_{y}$ are increasing, where
$\star^{r}_{y}(x)=x\star y$ and $\star^{l}_{y}(x)=y\star x$ for $x\in I$.
\end{lemma}

\begin{proof}
It follows from Proposition~\ref{Commu-Asso-Thm} that $\star^{r}_{y}=\star^{l}_{y}$. So,
it suffices to prove that $\star^{r}_{y}$ is increasing.

For any  $0\leq x_1\leq x_2\leq 1$, since $\curlywedge$ is increasing in each argument,
from Lemma~\ref{Lemma-Q}, it follows that
$$
\mathscr{V}_{x_1}\curlywedge \mathscr{V}_{y}\sqsubseteq \mathscr{V}_{x_2}\curlywedge \mathscr{V}_{y}.
$$
In particular, by Theorem \ref{order-theorem},
$$
(\mathscr{V}_{x_1}\curlywedge \mathscr{V}_{y})^{R}\leq (\mathscr{V}_{x_2}\curlywedge \mathscr{V}_{y})^{R}.
$$
This, together with Lemmas \ref{L-R} and \ref{Con-1}, implies that
\begin{align*}
x_{1}\star y& =\mathscr{V}_{x_1}(1)\star \mathscr{V}_{y}(1)
=(\mathscr{V}_{x_1}\curlywedge \mathscr{V}_{y})(1)\\
&=(\mathscr{V}_{x_1}\curlywedge \mathscr{V}_{y})^{R}(1)
\leq (\mathscr{V}_{x_2}\curlywedge \mathscr{V}_{y})^{R}(1)\\
&=(\mathscr{V}_{x_2}\curlywedge \mathscr{V}_{y})(1)=
\mathscr{V}_{x_2}(1)\star \mathscr{V}_{y}(1)\\
&=x_{2}\star y,
\end{align*}
i.e.,
$$
\star^{r}_{y}(x_1)=x_1\star y\leq x_{2}\star y=\star^{r}_{y}(x_2).
$$
Therefore, $\star^{r}_{y}$ is increasing.
\end{proof}

\subsection{Answer to Question \ref{Q-1}}

\begin{theorem}\label{Thm-20}
Let $\vartriangle$ be a continuous $t$-norm on $I$ and $\star$ be a binary operation on $I$.
If $\curlywedge$ is a $t$-norm on $\mathbf{L}$, then $\star$ is a $t$-norm.
\end{theorem}

\begin{proof}
It follows directly from Proposition~\ref{Commu-Asso-Thm}, and Lemmas \ref{1-elemrnt} and \ref{Increasing-Lemma}.
\end{proof}

\medskip

Similarly, the following result can be verified.

\medskip

\begin{theorem}\label{Thm-21}
Let $\triangledown$ be a continuous $t$-conorm on $I$ and $\star$ be a binary operation on $I$.
If $\curlyvee$ is a $t$-conorm on $\mathbf{L}$, then $\star$ is a $t$-norm.
\end{theorem}

\begin{remark}
Theorems \ref{Thm-20} and \ref{Thm-21} show that a binary operation $\star$ on $I$,
ensuring that $\curlywedge$ is a $t$-norm (thus a $t_{r}$-norm) on
$\mathbf{L}$ or $\curlyvee$ is a $t$-conorm (thus a $t_{r}$-conorm) on
$\mathbf{L}$, must be a $t$-norm. This give a negative answer to Question~\ref{Q-1}.
\end{remark}

\medskip

Combining together Theorems \ref{Thm-20}, \ref{Thm-21}, and \cite[Proposition~14]{HCT2015},
one obtains the following.

\medskip

\begin{theorem}
Let $\vartriangle$ be a continuous $t$-norm, $\triangledown$ be a continuous $t$-conorm, and
$\star$ be a continuous binary operation on $I$. Then, the following statements are equivalent:
\begin{enumerate}[(1)]
  \item $\star$ is a $t$-norm on $I$;
  \item $\curlywedge$ is a $t_{r}$-norm on $\mathbf{L}$;
  \item $\curlywedge$ is a $t$-norm on $\mathbf{L}$;
  \item $\curlyvee$ is a $t_{r}$-conorm on $\mathbf{L}$;
  \item $\curlyvee$ is a $t$-conorm on $\mathbf{L}$.
\end{enumerate}
\end{theorem}

\section{Further Study on the Binary Operation $\curlywedge$}
%%\begin{definition}\cite{HCT2015}
Let $\star$ be a continuous $t$-norm on $I$ and $\vartriangle$ be a surjective
binary operation on $I$. Define the binary operation
$\curlywedge: \mathbf{M}^2\rightarrow \mathbf{M}$ as follows:
for $f, g\in \mathbf{M}$,
\begin{equation}\label{e-3}
(f\curlywedge g)(x)=\sup\{f(y)\star g(z): y\vartriangle z =x\}.
\end{equation}
Here, the surjection assumption on $\vartriangle$ is necessary, because $(f\curlywedge g)(x)$
is not well defined for every point $x$ in $I\setminus \vartriangle(I^{2})$, if $\vartriangle$
is not surjective. Motivated by Question~\ref{Q-1}, a fundamental question is:
{\it Apart from the $t$-norms, does there exist other binary operation `$\vartriangle$' on
$I$ such that `$\curlywedge$' is a $t_r$-norm on $\mathbf{L}$?}

This section will also give a negative answer to this question.

\begin{lemma}\label{order-Lemma-2}
For $x_1, x_2\in I$, $\chi_{\{x_1\}}\sqsubseteq \chi_{\{x_2\}}$ if and only if $x_1\leq x_2$.
\end{lemma}

\begin{proof}
Firstly, it can be verified that, for any $x\in I$,
$$
\chi^{L}_{\{x\}}(t)=\begin{cases}
0, & t\in [0, x),\\
1, & t\in [x, 1],
\end{cases}
$$
and
$$
\chi^{R}_{\{x\}}(t)=\begin{cases}
1, & t\in [0, x],\\
0, & t\in (x, 1].
\end{cases}
$$
This, together with Theorem~\ref{order-theorem}, implies that
\begin{eqnarray*}
&&\chi_{\{x_1\}}\sqsubseteq \chi_{\{x_2\}}\\
&\Leftrightarrow& \chi^{L}_{\{x_2\}} \leq \chi^{L}_{\{x_1\}} \text{ and } \chi^{R}_{\{x_1\}} \leq \chi^{R}_{\{x_2\}}\\
&\Leftrightarrow& x_1\leq x_2.
\end{eqnarray*}
\end{proof}

\begin{lemma}\label{operation-Lemma}
Let $\star$ be a continuous $t$-norm on $I$ and $\vartriangle$ be a binary operation on $I$.
Then,
for any $x_1, x_2\in I$, $\chi_{\{x_1\}}\curlywedge \chi_{\{x_2\}}=\chi_{\{x_1\vartriangle x_2\}}$.
\end{lemma}

\begin{proof}
Since $\star$ is a continuous $t$-norm, applying Lemma~\ref{1-Lemma} gives
\begin{enumerate}[(a)]
  \item for $y, z\in I$, $\chi_{\{x_1\}}(y)\star \chi_{\{x_2\}}(z)\in\{0, 1\}$;
  \item $\chi_{\{x_1\}}(y)\star \chi_{\{x_2\}}(z)=1$ if and only if $y=x_1$ and $z=x_2$.
\end{enumerate}
This, together with $(\chi_{\{x_1\}}\curlywedge \chi_{\{x_2\}})(x)
=\sup\{\chi_{\{x_1\}}(y)\star \chi_{\{x_2\}}(z): y\vartriangle z=x\}$, implies that
$$
\chi_{\{x_1\}}\curlywedge \chi_{\{x_2\}}=\chi_{\{x_1\vartriangle x_2\}}.
$$
\end{proof}

\begin{lemma}\label{Commu-Asso-Thm-2}
Let $\star$ be a continuous $t$-norm on $I$ and $\vartriangle$ be a binary operation on $I$.
Then, \begin{enumerate}[(1)]
  \item $\curlywedge$ is commutative on $\mathbf{L}$ if and only if $\vartriangle$ is commutative;
  \item If $\curlywedge$ is associative on $\mathbf{L}$, then $\vartriangle$ is
associative.
\end{enumerate}
\end{lemma}
\begin{proof}
(1) The sufficiency holds trivially. It remains to check the necessity.

For $x, y\in I$, since $\curlywedge$ is commutative, it follows from
Lemma~\ref{operation-Lemma} that
$$
\chi_{\{x\vartriangle y\}}=\chi_{\{x\}}\curlywedge \chi_{\{y\}}
=\chi_{\{y\}}\curlywedge \chi_{\{x\}}=\chi_{\{y\vartriangle x\}},
$$
implying that
$$
x\vartriangle y=y\vartriangle x.
$$
Thus,
$\vartriangle$ is commutative.

(2) For $x, y, z\in I$, since $\curlywedge$ is associative, it follows from
Lemma~\ref{operation-Lemma} that
\begin{align*}
\chi_{(x\vartriangle y)\vartriangle z}&=\chi_{\{x\vartriangle y\}}\curlywedge \chi_{\{z\}}=(\chi_{\{x\}}\curlywedge \chi_{\{y\}})\curlywedge \chi_{\{z\}}\\
&=\chi_{\{x\}}\curlywedge (\chi_{\{y\}}\curlywedge \chi_{\{z\}})=\chi_{\{x\}}\curlywedge\chi_{\{y\vartriangle z\}}\\
&=\chi_{x\vartriangle (y\vartriangle z)},
\end{align*}
implying that
$$
(x\vartriangle y)\vartriangle x=x\vartriangle (y\vartriangle z).
$$
Thus,
$\vartriangle$ is associative.
\end{proof}

\begin{lemma}\label{1-elemrnt-2}
Let $\star$ be a continuous $t$-norm on $I$ and $\vartriangle$ be a binary operation on $I$.
If $\curlywedge$ is a $t$-norm on $\mathbf{L}$, then
$1\vartriangle x=x\vartriangle 1=x$ for all $x\in I$.
\end{lemma}

\begin{proof}
For any $x\in I$, since $\chi_{\{1\}}$ is an neutral element,
applying Lemma \ref{operation-Lemma} yields that
$$\chi_{\{1\vartriangle x\}}=\chi_{\{1\}}\curlywedge \chi_{\{x\}}=\chi_{\{x\}}=\chi_{\{x\}}
\curlywedge \chi_{\{1\}}=\chi_{\{x\vartriangle 1\}}.$$
Thus,
$1\vartriangle x=x=x\vartriangle 1$.
\end{proof}

\begin{lemma}\label{Increasing-Lemma-2}
Let $\star$ be a continuous $t$-norm on $I$ and $\vartriangle$ be a binary operation on $I$.
If $\curlywedge$ is a $t$-norm on $\mathbf{L}$, then, for any $y\in (0, 1)$,
the functions $\vartriangle^{r}_{y}$ and $\vartriangle^{l}_{y}$ is increasing, where
$\vartriangle^{r}_{y}(x)=x\vartriangle y$ and $\vartriangle^{l}_{y}(x)=y\vartriangle x$ for any $x\in I$.
\end{lemma}

\begin{proof}
It follows from Lemma~\ref{Commu-Asso-Thm-2} that $\vartriangle^{r}_{y}=\vartriangle^{l}_{y}$.
So, it suffices to prove that $\vartriangle^{r}_{y}$ is increasing.

For $0\leq x_1\leq x_2\leq 1$,
applying Lemma~\ref{order-Lemma-2} follows that
$$
\chi_{\{x_1\}}\sqsubseteq \chi_{\{x_2\}}.
$$
Since $\curlywedge$ is increasing in each argument, applying Lemma~\ref{operation-Lemma} yields that
$$
\chi_{\{x_1\vartriangle y\}}=\chi_{\{x_1\}}\curlywedge \chi_{\{y\}}\sqsubseteq \chi_{\{x_2\}}\curlywedge \chi_{\{y\}}
=\chi_{\{x_2\vartriangle y\}}.
$$
This, together with Lemma~\ref{order-Lemma-2}, implies that
$$
\vartriangle^{R}_{y}(x_1)=x_1\vartriangle y \leq x_2\vartriangle y=\vartriangle^{R}_{y}(x_2).
$$
Therefore, $\vartriangle^{R}_{y}$ is increasing.
\end{proof}

\medskip

Combining together Lemmas \ref{Commu-Asso-Thm-2}--\ref{Increasing-Lemma-2} and \cite[Proposition~14]{HCT2015}
immediately yields the following result.

\medskip

\begin{theorem}
Let $\star$ be a continuous $t$-norm on $I$ and $\vartriangle$ be
a continuous binary operation on $I$. Then, the following statements are equivalent:
\begin{enumerate}[(1)]
  \item $\vartriangle$ is a $t$-norm on $I$;
  \item $\curlywedge$ is a $t_{r}$-norm on $\mathbf{L}$;
  \item $\curlywedge$ is a $t$-norm on $\mathbf{L}$.
\end{enumerate}
\end{theorem}

Similarly, one can obtain the following result.
\begin{theorem}
Let $\star$ be a continuous $t$-norm on $I$ and $\vartriangle$ be
a continuous binary operation on $I$. Then, the following statements are equivalent:
\begin{enumerate}[(1)]
  \item $\vartriangle$ is a $t$-conorm on $I$;
  \item $\curlywedge$ is a $t_{r}$-conorm on $\mathbf{L}$;
  \item $\curlywedge$ is a $t$-conorm on $\mathbf{L}$.
\end{enumerate}
\end{theorem}

\section{Conclusion}
This paper has further studied the binary operations $\curlywedge$ and $\curlyvee$ defined in \eqref{O-1},
\eqref{O-2} and \eqref{e-3} on $\mathbf{L}$. By introducing two special families of
functions $\mathscr{W}_{x}$ and $\mathscr{V}_{x}$ ($x\in I$),
it first proves that, when the continuous $t$-norm $\vartriangle$ or continuous
$t$-conorm $\triangledown$ is fixed, the following hold:
\begin{enumerate}[(1)]
  \item $\curlywedge$ is a continuous $t_r$-norm on $\mathbf{L}$ if and only if $\curlywedge$ is a continuous $t$-norm on $\mathbf{L}$
if and only if $\star$ is a continuous $t$-norm;
  \item $\curlyvee$ is a continuous $t_r$-conorm on $\mathbf{L}$ if and only if $\curlyvee$ is a continuous $t$-conorm on $\mathbf{L}$
if and only if $\star$ is a continuous $t$-norm.
\end{enumerate}
In particular, these results negatively answer Question~\ref{Q-1}. Similarly to
Question~\ref{Q-1}, the case that the binary operation $\vartriangle$ is fixed (see \eqref{e-3})
has been investigated and some analogous results were obtained.

%%\section*{Acknowledgment}

\end{document}